\documentclass[english,11pt,leqno,twoside]{amsart}

\usepackage{amssymb,amsmath,amsthm,soul,color}
\usepackage{graphicx}
\usepackage[noadjust]{cite}
\usepackage{t1enc}
\usepackage[utf8]{inputenc}

\usepackage{braket}
\usepackage[T1]{fontenc}

\usepackage{amsfonts}
\usepackage{a4,indentfirst,latexsym,graphics}
\usepackage[english]{babel}

\usepackage[metapost]{mfpic}
\usepackage[colorinlistoftodos]{todonotes}
\usepackage{mathtools} 

\usepackage{fixmath}
\usepackage{bbm}

\usepackage[left=2.1cm,right=2.1cm,top=2cm,bottom=2cm]{geometry}

\usepackage{url}

\usepackage{mathrsfs}

\usepackage{enumitem}
\usepackage[colorlinks=true,urlcolor=blue,
citecolor=red,linkcolor=blue,linktocpage,pdfpagelabels,
bookmarksnumbered,bookmarksopen]{hyperref}

\usepackage[hyperpageref]{backref}

\usepackage[normalem]{ulem}

\usepackage{lmodern}

\newtheorem{Th}{Theorem}[section]
\newtheorem{Prop}[Th]{Proposition}
\newtheorem{Lem}[Th]{Lemma}

\newtheorem{Def}[Th]{Definition}
\newtheorem{Rem}[Th]{Remark}

\newcommand{\R}{\mathbb{R}}

\newcommand{\bR}{\mathbb{R}}

\newcommand{\cC}{{\mathcal C}}
\newcommand{\cD}{{\mathcal D}}

\numberwithin{equation}{section}

\DeclareMathOperator*{\essinf}{ess\,inf}

\newcommand{\twostar}{{2^{\ast}}}

\newcommand{\dvg}{\mathrm{d}v_g}

\newcommand{\FF}{\mathbf{F}}
\renewcommand{\AA}{\mathcal{L}}

\newcommand{\upz}{u_0(p_0)}

\newcommand{\fa}{\mathfrak{a}}
\newcommand{\fb}{\mathfrak{b}}
\newcommand{\fc}{\mathfrak{c}}
\newcommand{\fd}{\mathfrak{d}}
\newcommand{\fh}{\mathfrak{h}}

\everymath{\displaystyle}

\title[A Lichnerowicz equation on non-CMC closed manifolds]{A Lichnerowicz equation in the Einstein-scalar field theory on non-CMC closed manifolds}

\author[B. Bieganowski]{Bartosz Bieganowski}
\author[P. d'Avenia]{Pietro d'Avenia}
\author[J. Schino]{Jacopo Schino}
\author[D. Strzelecki]{Daniel Strzelecki}

\address[B. Bieganowski]{\newline\indent
	Faculty of Mathematics, Informatics and Mechanics, \newline\indent
	University of Warsaw, \newline\indent
	ul. Banacha 2, 02-097 Warsaw, Poland}
\email{\href{mailto:bartoszb@mimuw.edu.pl}{bartoszb@mimuw.edu.pl}}

\address[P. d'Avenia]{\newline\indent
	Dipartimento di Meccanica, Matematica e Management, \newline\indent
	Politecnico di Bari, \newline\indent
	Via E. Orabona 4, 70125 Bari, Italy}
\email{\href{mailto:pietro.davenia@poliba.it}{pietro.davenia@poliba.it}}

\address[J. Schino]{\newline\indent
	Faculty of Mathematics, Informatics and Mechanics, \newline\indent
	University of Warsaw, \newline\indent
	ul. Banacha 2, 02-097 Warsaw, Poland
    \newline\indent
    and
    \newline\indent
    Neuromed Mediterranean University of Sciences and Technologies
    \newline\indent
    Via dell'Elettronica 1, 86077 Pozzilli (IS), Italy}
\email{\href{mailto:jacopo.schino@unineuromed.it}{jacopo.schino@unineuromed.it}}

\address[D. Strzelecki]{\newline\indent
	Faculty of Mathematics, Informatics and Mechanics, \newline\indent
	University of Warsaw, \newline\indent
	ul. Banacha 2, 02-097 Warsaw, Poland
    \newline\indent
    and
    \newline\indent
    Faculty of Mathematics and Computer Science
    \newline\indent
    Nicolaus Copernicus University in Toruń
    \newline\indent
    ul. Chopina 12/18, 87-100 Toruń, Poland
    }
\email{\href{mailto:dstrzelecki@mimuw.edu.pl}{dstrzelecki@mimuw.edu.pl}}

\begin{document}

\begin{abstract} 
In the paper, we prove the existence of a positive and essentially bounded solution to a Lichnerowicz equation in the Einstein-scalar field theory on a closed manifold with non-constant mean curvature. In particular, the non-constant mean curvature gives rise to supercritical terms in the equation, on top of singular ones. We employ a recent fixed-point argument, which involves sub- and supersolutions. Additionally, we provide several conditions on the coefficients in the equation that prevent the existence of positive classical solutions.
\end{abstract}

\keywords{Fixed point arguments, singular nonlinearities, Einstein field equations, Lichnerowicz
equation, elliptic problems.}
\subjclass[2020]{58J05,   	
    58J90, 
    35J75, 
    35J61, 
    35Q75, 
    35Q76
    }

\maketitle

\section{Introduction}
To study the Einstein equation in general relativity in a four-dimensional spacetime $\mathcal{V}$ as a Cauchy problem, it is necessary to find appropriate initial data. The Gauss and Codazzi equations, when applied to the Einstein equation, provide the conditions for a metric $g$, a symmetric $(0,2)$-tensor $K$, and a mean curvature $\tau$ in a three-dimensional space manifold $M$ that must be satisfied to ensure that $M$ is the initial state for the study of evolution. These conditions are known as \emph{constraint equations} (for a detailed derivation, see \cite{Isenberg&Bartnik}).

One of the most fruitful techniques in studying constraint equations is a conformal method (also called ``Method B'') introduced by Lichnerowicz \cite{Lichnerowicz}. This method allows for the classification of the constant mean curvature (CMC) solutions of such equations on compact manifolds (see \cite{Isenberg95}) or other types of initial manifolds (see \cite{Maxwell-2009} and references therein). 

The situation becomes more complicated when studying non-constant mean curvature (non-CMC) solutions. This was not a case of interest for many years, since it was thought that each solution of the Einstein equation had a slice with constant mean curvature. In 1988, Bartnik provided an example of a spacetime satisfying the Einstein equation that does not admit any CMC surface, see \cite{Bartnik88}. A general method to create such examples was given in \cite{Chrusciel}. Therefore, the study of non-CMC solutions proved to be an important topic.

In the study of the Einstein field equation, the coupling of gravity with other fields provides the form of the energy-momentum tensor. For instance, this approach leads to the Einstein–Maxwell equations for electromagnetism in curved spacetime.

In this paper, we focus on the so-called Einstein-scalar field theory: the coupling of gravity with a scalar field.
Let $\Psi$ be a real-valued scalar field on an $(N+1)$-dimensional spacetime manifold $\mathcal{V}$ ($N \ge 3$) and let $\gamma$ be a spacetime metric. Consider the action
\[
\mathcal{S}(\Psi,\gamma):=\int_\mathcal{V} \Big[R(\gamma)-\frac{1}{2}|\nabla\Psi|_{\gamma}^2-V(\Psi)\Big] \mathrm{d}\eta_{\gamma},
\]
where $R(\gamma)$ is the scalar curvature of $\gamma$, $\mathrm{d}\eta_{\gamma}$ is the volume element, $|\nabla\Psi|_{\gamma}^2$ is the squared pseudo-norm of the spacetime gradient of the field $\Psi$ with respect to the metric $\gamma$, and $V\colon\R\to\R$ is a given smooth function.\footnote{For the Klein–Gordon field theory, $V(\Psi) = \frac{1}{2}m^2\Psi^2$.}
Making the variations of $\mathcal{S}$ with respect to $\gamma$ and $\Psi$, we get the Einstein-scalar field equations
\[
\begin{cases}
G_{\alpha\beta}=\nabla_\alpha\Psi\nabla_\beta \Psi -\frac{1}{2}\gamma_{\alpha\beta}\nabla_\mu\Psi\nabla^\mu \Psi - \gamma_{\alpha\beta}V,\\
\nabla_\mu\Psi\nabla^\mu \Psi=\frac{\mathrm{d}V}{\mathrm{d}\Psi}
\end{cases}
\]
where $G_{\alpha\beta}$ is the Einstein curvature tensor.
The same system was obtained by Rendall \cite{Rendall2} to describe the expansion of the universe (see also \cite{Rendall}). Moreover, in particle physics, scalar fields are used to describe spin-zero particles.

Now, let us consider the constraint equations for this problem, which are conditions for an $N$-dimensional spacelike hypersurface $M$ with a spatial metric $g$, a symmetric tensor $K$, and a mean curvature $\tau$. As noted in \cite{Choquet-Bruhat_2007}, the coupling of a scalar field to the Einstein gravitational field theory does not add any new constraint equations to the theory. In the same paper, Choquet-Bruhat, Isenberg, and Pollack applied  the Lichnerowicz conformal method to reformulate the constraint equations for the Einstein-scalar system into a system of partial differential equations.

In this method, $\Psi$ is split into $\psi$ (its restriction to $M$) and $\pi$ (the normalized time derivative of $\Psi$ restricted to $M$). Introducing the new variables $u$ for a positive scalar function and $W$ for a vector field defined on $M$, one can rescale and transform $\pi$,
$K$, and $g$ (we keep these symbols for the transformed variables) to rewrite the constraint equations as the system

\begin{equation}\label{E:constraints_full_eqaution}
\left\{
  \begin{array}{l}
    \begin{aligned}
      -\Delta_g u &+ \frac{N-2}{4(N-1)} \left(R(g) - |\nabla_g \psi|_g^2\right) u
        - \frac{N-2}{4(N-1)}\left(|\sigma + \cD W|_g^2 + \pi^2\right) u^{-(2^*+1)}\\
      \quad &+ \frac{N-2}{4(N-1)}\left(\frac{N-1}{N} \tau^2 - 4 V(\psi)\right) u^{2^*-1}
        = 0,
    \end{aligned}
    \\[1ex]
    \mathrm{div}_g(\cD W)
      = \frac{N-1}{N} u^{2^*} \nabla_g\tau - \pi \nabla_g\psi,
  \end{array} \text{in } M,
\right.
\end{equation}
where $2^*:=\frac{2N}{N-2}$ denotes the critical Sobolev exponent.

In this system, $R(g)$ is the scalar curvature (also called the Ricci scalar), $\sigma$ is the so-called TT-tensor (symmetric, divergence free and trace free), $\cD$ is the conformal Killing operator, $\tau$ represents the mean curvature of $M$, and $V$ is the smooth potential of the Einstein-scalar field theory. 

The standard task in this approach is to verify which choices of the data $(g,\sigma,\tau,\psi,\pi)$ allow solving the constraint equations for $u$ and $W$. Then, from the full set of data $(g,\sigma,\tau,\psi,\pi,W, u)$, by the reverse rescaling, one can obtain the solution of the constraint equations.


In \cite{Choquet-Bruhat_2007} and \cite{Hebey2008}, the authors consider the system $\eqref{E:constraints_full_eqaution}$ in the case where $\tau$ is constant (see also \cite{BK}, where the authors admit more general nonlinearities). In that scenario, the second equation becomes independent of $u$. By solving it for $W$, the problem reduces to studying the first equation for the function $u$, whose nonlinear part consists of Sobolev-critical and negative powers. We also mention \cite{Chrusciel_Gicquaud,PremoselliCVPD} and \cite{Hebey_Veronelli,PremoselliDCDS}, which address, respectively, multiplicity of solutions and asymptotic behavior.

In this paper, from a mathematical perspective, we deal with non-constant $\tau$, making the system \eqref{E:constraints_full_eqaution} fully coupled; in this way, a different point of view is taken on the variables in the constraint equation. We fix the vector field $W$ and regard $u$ and $\psi$ as unknowns. This approach, as will be shown, leads then to a single equation for $u$, independent of $\psi$, with Sobolev-supercritical terms, and we focus on it (i.e., we are interested only in $u$).

The choice of $W$ is free in the sense that it is not determined by an equation. However, it is not completely arbitrary. First, $W$ has to be sufficiently regular for the coefficients of the resulting equation \eqref{eq:1} below to satisfy our assumptions, e.g., $|\cD W|_g, \mathrm{div}_g(\cD W) \in L^\infty(M)$ in Section~\ref{sec:ex}. Moreover, the sign condition $\langle C, D \rangle_g \ge 0$ in (\ref{A:coeff_properties}) and in Section~\ref{sec:non} is a restriction on $W$, namely $\langle \mathrm{div}_g(\cD W), \nabla_g \tau \rangle_g \le 0$ on $M$; it is satisfied, e.g., when $W$ is a conformal Killing field ($\cD W = 0$). Finally, since $\pi^{-1}$ appears in the coefficients, we need $\pi$ to be bounded away from zero, and recovering $\psi$ from the second equation of \eqref{E:constraints_full_eqaution} requires the vector field $\pi^{-1}\bigg(\frac{N-1}{N} u^{2^*}\nabla_g\tau - \mathrm{div}_g (\cD W)\bigg)$ to be a gradient on $M$; this compatibility condition is not addressed here, as we focus on the equation for $u$.

From now on, we assume that $(M,g)$ is a closed (i.e., compact and without boundary) Riemannian manifold of dimension $N \ge 3$.
We assume $V \equiv \nu\in\mathbb{R}$. The case of a constant potential can be understood as the situation when the scalar field does not produce a force (the force is the gradient of a potential). Especially, assuming $V\equiv 0$, we study the scalar field of massless particles, for example, the massless Klein--Gordon equation.

Then, for a general $\tau$, possibly non-constant, from the second equation we get

$$
|\nabla_g \psi|_g^2 = \left| \pi^{-1} \left( \frac{N-1}{N} u^{2^*}\nabla_g\tau - \mathrm{div}_g (\cD W) \right) \right|_g^2.
$$
Hence, replacing it in the first equation of the system \eqref{E:constraints_full_eqaution}, $\psi$ disappears, giving the following\linebreak Lichnerowicz-type equation
\begin{align*}
&-\Delta_g u + \frac{N-2}{4(N-1)} \left( R(g) - \left| \pi^{-1} \left( \frac{N-1}{N} u^{2^*}\nabla_g\tau - \mathrm{div}_g (\cD W) \right) \right|_g^2 \right) u \\
&\qquad - \frac{N-2}{4(N-1)} ( | \sigma + \cD W |_g^2 + \pi^2 ) u^{-(2^*+1)} + \frac{N-2}{4(N-1)}\left(\frac{N-1}{N} \tau^2 - 4 \nu\right)  u^{2^*-1} = 0.
\end{align*}
We set
\begin{align*}
h &:= \frac{N-2}{4(N-1)} R(g), \quad A := \frac{N-2}{4(N-1)} ( | \sigma + \cD W |_g^2 + \pi^2 ),\quad B := \frac{N-2}{4(N-1)}\left(\frac{N-1}{N} \tau^2 - 4 \nu\right),\\
C &:= - \sqrt{\frac{N-2}{4(N-1)}} \pi^{-1} \mathrm{div}_g (\cD W), \quad D:= \sqrt{\frac{N-2}{4(N-1)}} \pi^{-1} \frac{N-1}{N} \nabla_g\tau.
\end{align*}
Then we rewrite the Lichnerowicz equation as
$$
-\Delta_g u + (h - |D u^{2^*} + C|_g^2)u - A u^{-(2^*+1)} + B u^{2^*-1} = 0
$$
or equivalently
\begin{equation}\label{eq:1}
-\Delta_g u + (h-|C|_g^2) u = |D|_g^2 u^{2 \cdot 2^* + 1} + 2 \langle C , D \rangle_g u^{2^*+1} - B u^{2^*-1} + \frac{A}{u^{2^*+1}} \quad\text{in } M,
\end{equation}
which is the equation that we study in this paper. Note that $A$, $B$, $C$, $D$, $h$ are not necessarily constant. In fact, $A,B,h\colon M\to\bR$ and $C,D \colon M\to TM$.

The article is organized as follows. In Section \ref{sec:ex}, we prove the existence of a positive solution to \eqref{eq:1} using a sub/supersolution method followed by the fixed point theorem for monotone operators by Carl and Heikkilä \cite{Carl2002}, which is an extension of the Knaster--Tarski theorem \cite{Tarski}. The main result is formulated in Theorem \ref{th:Main}. We highlight that the presence of supercritical terms makes it unclear how to use a variational approach as in \cite{BK,Hebey2008}.
In Section \ref{sec:non}, we provide sufficient conditions for the nonexistence of $\cC^2$ positive solutions to \eqref{eq:1}.

In what follows, $K$ denotes a generic positive constant that may vary from one line to another.

\section{Existence of a weak solution}\label{sec:ex}

Recall that $M$ is assumed to be an $N$-dimensional closed manifold. To prove the existence of weak solutions of the problem \eqref{eq:1}, we work under the following assumptions:
\begin{enumerate}[label=(A\arabic{*}),ref=A\arabic{*}]\setcounter{enumi}{0}
    \item \label{A:coeff_in_proper_spaces}$A,B, |C|_g,|D|_g, h\in L^\infty(M)$;
     \item \label{A:coeff_properties} $\essinf_M A >0$, $\langle C,D\rangle_g\geq 0$;
    \item \label{A:positive_operator} $\essinf_M (h - |C|_g^2) > -\lambda_1$, where $\lambda_1 > 0$ is the first eigenvalue of $-\Delta_g$ on $M$;
    \item \label{A:essinf_of_h} $\essinf_{M} h >\min_{t>0} r(t)$, where 
  $r(t):=\left(\|C\|_{L^\infty}+\|D\|_{L^\infty} t^{\twostar} \right)^2 - \essinf_M B \cdot t^{2^*-2} + \frac{\|A\|_{L^\infty}}{t^{\twostar+2}} $.
\end{enumerate}
Here, for a measurable function $f\colon M\to\R$,
$\essinf_M f$ denotes the essential infimum of $f$.

From now on, we assume that the assumptions (\ref{A:coeff_in_proper_spaces})--(\ref{A:essinf_of_h}) are satisfied.
By $H^1(M)$ we denote the Sobolev space obtained as the completion of $\mathcal{C}^\infty(M)$ with respect to the norm $\|u\|_{H^1(M)}^2 := \int_M |\nabla_g u|_g^2 + u^2 \, \dvg$. It is a Hilbert space, and $H^{-1}(M)$ denotes its dual space.
Below, we define what we mean by weak solution, subsolution and supersolution of our problem.

\begin{Def}
Let $u \in H^1(M) \cap L^\infty(M)$. We say that $u$ is a weak subsolution to \eqref{eq:1} if and only if for every nonnegative $\varphi \in H^1(M)$, there holds $\int_M \frac{A \varphi}{|u|^{2^*+1}} \, \dvg < +\infty$ and

\begin{equation*}
\begin{split}
&\int_M \langle \nabla_g u , \nabla_g \varphi \rangle_g + (h - |C|^2_g) u \varphi \, \dvg\\
&\qquad\le \int_M |D|^2_g |u|^{2\cdot \twostar} u \varphi + 2 \langle C , D \rangle_g |u|^{2^*} u \varphi - B |u|^{\twostar - 2} u \varphi + \frac{A u \varphi}{|u|^{2^*+2}} \, \dvg.
\end{split}
\end{equation*}

Similarly, we say that $u$ is a weak supersolution if and only if $-u$ is a weak subsolution, and that $u$ is a weak solution if and only if it is both a weak subsolution and a weak supersolution.
\end{Def}

The main result of this paper is the following theorem.

\begin{Th}\label{th:Main}
Suppose that (\ref{A:coeff_in_proper_spaces})--(\ref{A:essinf_of_h}) are satisfied. Then \eqref{eq:1} has a weak solution which is positive and bounded away from zero almost everywhere on $M$. If, moreover, the functions in (\ref{A:coeff_in_proper_spaces}) and $\langle C, D\rangle_g$ are H\"older continuous, then this solution is of class $\cC^{2,\alpha}$ for some $\alpha<1$.
\end{Th}

To prove Theorem \ref{th:Main}, we first formulate and prove a series of lemmas. In the first one, we prove the existence of a small constant weak subsolution.

\begin{Lem}
There exists $\delta_0$ such that for all $0 < \vartheta < \delta_0$, the constant function $\vartheta$ is a weak subsolution to \eqref{eq:1}.
\end{Lem}
\begin{proof}
Note that, for a sufficiently small constant $\vartheta > 0$, we have
\begin{align*}
  \vartheta^{2\cdot 2^*}  |D|_g^2  + 2 \vartheta^{2^*}  \langle C, D \rangle_g  - \vartheta^{2^*-2}  B  + \frac{1}{\vartheta^{2^* + 2}}  A
  &
  \geq - \vartheta^{2^*-2}  B  + \frac{1}{\vartheta^{2^* + 2}}  A\\
  & \geq - \vartheta^{2^*-2}  \|B\|_{L^\infty(M)}  + \frac{1}{\vartheta^{2^* + 2}}  \essinf_M A\\
  & \geq \| h\|_{L^\infty(M)} - \essinf_M |C|_g^2\\
  &\geq h-|C|_g^2
\end{align*}
almost everywhere on $M$. Thus, multiplying it by any nonnegative function $\varphi \in H^1(M)$ and integrating, we complete the proof.
\end{proof}

In the subsequent lemma, we prove the existence of a constant weak supersolution.

\begin{Lem}
There exists a constant weak supersolution $\Theta$ to \eqref{eq:1}.
\end{Lem}

\begin{proof}
Take $\Theta>0$ to be a minimizer of the function $r$ defined in (\ref{A:essinf_of_h}).
Then, thanks to (\ref{A:essinf_of_h}),
\begin{align*}
h - |C|_g^2
& \geq \left( \|C\|_{L^\infty} + \|D\|_{L^\infty} \Theta^{2^*} \right)^2 -\essinf_M B\cdot\Theta^{2^*-2}+ \frac{\|A\|_{L^\infty}}{\Theta^{2^*+2}}-|C|_g^2 \\
&= \|C\|_{L^\infty}^2 -|C|_g^2+ \|D\|_{L^\infty}^2 \Theta^{2 \cdot 2^*}+2\|C\|_{L^\infty} \|D\|_{L^\infty} \Theta^{2^*}  -\essinf_M B\cdot \Theta^{2^*-2} + \frac{\|A\|_{L^\infty}}{\Theta^{2^*+2}}  \\
&\geq |D|_g^2 \Theta^{2 \cdot 2^*} + 2\langle C, D\rangle_g \Theta^{2^*}- B \Theta^{2^*-2} + \frac{A}{\Theta^{2^*+2}},
\end{align*}
and multiplying the inequality by any $\varphi \in H^1 (M)$ with $\varphi \geq 0$, we conclude.
\end{proof}

From now on, we fix any constant weak subsolution $\vartheta$ such that $\vartheta < \Theta$. We introduce the truncated functions $f_1,f_2 \colon M \times \R \rightarrow \R$ by
$$
f_1(p, \xi) := 
\begin{cases}
|D(p)|_g^2 \vartheta^{2 \cdot 2^* + 1} + 2 \langle C(p) , D(p) \rangle_g \vartheta^{2^*+1} + B^-(p) \vartheta^{2^*-1}     & \mbox{ if } \xi < \vartheta\\
|D(p)|_g^2 \xi^{2 \cdot 2^* + 1} + 2 \langle C(p) , D(p) \rangle_g \xi^{2^*+1} + B^-(p) \xi^{2^*-1}    & \mbox{ if } \vartheta \leq \xi \leq \Theta\\
|D(p)|_g^2 \Theta^{2 \cdot 2^* + 1} + 2 \langle C(p) , D(p) \rangle_g \Theta^{2^*+1} + B^-(p) \Theta^{2^*-1}    & \mbox{ if } \xi > \Theta
\end{cases} 
$$
and
$$
f_2(p, \xi) :=
\begin{cases}
B^+(p) \vartheta^{2^*-1} - \frac{A(p)}{\vartheta^{2^*+1}}     & \mbox{ if } \xi < \vartheta\vspace{3pt}\\
B^+(p) \xi^{2^*-1} - \frac{A(p)}{\xi^{2^*+1}}     & \mbox{ if } \vartheta \leq \xi \leq \Theta\vspace{3pt}\\
B^+(p) \Theta^{2^*-1} - \frac{A(p)}{\Theta^{2^*+1}}     & \mbox{ if } \xi > \Theta
\end{cases}, 
$$
where $B^-(p)=-\min\{B(p),0\}$ and $B^+(p)=\max\{B(p),0\}$, and we consider the truncated problem
\begin{equation}\label{eq:truncated}
-\Delta_g u + (h - |C|_g^2) u +f_2(p,u)= f_1(p, u), \quad u \in H^1(M).
\end{equation}
It is clear that, due to (\ref{A:coeff_in_proper_spaces}),
$f_i(\cdot,u(\cdot)) \in H^{-1}(M)$. Thus, we introduce the Nemytskii operators associated with $f_i$ (see, e.g., \cite{Struwe}), namely, $\FF_i \colon H^1(M) \rightarrow H^{-1} (M)$ given by
$$
[\FF_i(u)](v) := \int_{M} f_i(\cdot, u(\cdot)) v \, \dvg.
$$

\noindent On $H^1(M)$ we consider the partial order
\[
u_1\leq u_2 \ \Longleftrightarrow  \ (u_1-u_2)^+=0\quad\text{a.e. on } M.
\]
Then on $H^{-1}(M)$ we define the partial order by
\[
h_1 \le h_2 \
\Longleftrightarrow \ h_1(v) \leq h_2(v) \quad \forall v \in H^1(M) \text{ such that }v \geq 0.
\]
We say that an operator $\FF \colon H^1(M) \rightarrow H^{-1}(M)$ is increasing if and only if
$$
u_1 \leq u_2 \quad \Longrightarrow \quad \FF(u_1) \leq \FF(u_2).
$$
In the spirit of this definition we have the following.

\begin{Lem}
$\FF_i$ are increasing operators.
\end{Lem}

\begin{proof}
Note that by the assumption (\ref{A:coeff_properties}), for almost every $p\in M$ and $\xi_1,\xi_2\in \bR$ such that $\xi_2\geq \xi_1$ we have
\[
f_i(p,\xi_2)\geq f_i(p,\xi_1),\quad i=1,2.
\]
Let $u_1,u_2\in H^1(M)$ be such that $u_2\geq u_1$. Then
\[
[\FF_i(u_2)](v)-[\FF_i(u_1)](v)=\int_M \left(f_i(\cdot,u_2(\cdot))-f_i(\cdot,u_1(\cdot))\right) v \, \dvg\geq 0,\quad i=1,2,
\]
and the proof is complete.
\end{proof}

Define the operator $\AA \colon H^1(M) \rightarrow H^{-1}(M)$ by
$$
[\AA(u)](v) = \int_{M} \nabla_g u \nabla_g v + (h-|C|_g^2) uv \, \dvg, \quad u,v \in H^1(M).
$$
Note that (\ref{A:positive_operator}) implies that
\begin{equation}\label{eq:A-positive-definite-kappa}
[\AA (u)](u)\geq \kappa \|u\|^2_{H^1(M)}, \quad u \in H^1(M),
\end{equation}
for some constant $\kappa > 0$.

Observe that weak solutions of \eqref{eq:truncated} are functions $u \in H^1(M)$
such that
\begin{equation}\label{eq:functional_equation}
\AA(u) + \FF_2(u) = \FF_1(u).
\end{equation}

\begin{Lem}
The operator $\AA+\FF_2 \colon H^1(M)\to H^{-1}(M)$ is bijective. Furthermore, its inverse \linebreak $(\AA+\FF_2)^{-1}\colon H^{-1}(M)\to H^1(M)$ is increasing.
\end{Lem}
\begin{proof}
It is clear that $\AA$ is linear and continuous. We will show that $\FF_2$ is continuous. Take any sequence $(u_n) \subset H^1(M)$ such that $u_n \to u$ in $H^1(M)$. Let $v \in H^1(M)$ with $\|v\|_{H^1} \le 1$. From the Sobolev embedding, there exists $K > 0$ independent of $v$ or $n$ such that, using H\"older's inequality,
\begin{equation*}
\begin{split}
\left| [\FF_2(u_n) - \FF_2(u)](v) \right| & \le \int_{M} \left| f_2(p,u_n(p)) - f_2(p,u(p)) \right| |v| \, \dvg\\
&\le K \left( \int_{M} \left| f_2(p,u_n(p)) - f_2(p,u(p)) \right|^{2} \, \dvg \right)^{1/2}.
\end{split}
\end{equation*}
Since $f_2$ is continuous in the second variable, $f_2(p,u_n(p)) \to f_2(p,u(p))$ for a.e. $p \in M$. In addition, there exists $K > 0$ independent on $n$ such that
\begin{equation*}
\left| f_2(\cdot,u_n(\cdot)) - f_2(\cdot,u(\cdot)) \right|^2 \le K \left( (B^+)^{2} + A^{2} \right) \in L^1(M)
\end{equation*}
in view of (\ref{A:coeff_in_proper_spaces}). 
Then the Lebesgue dominated convergence theorem allows us to conclude.

One can see that, for all $u,v\in H^1(M)$, there holds
$[\FF_2(u)-\FF_2(v)](u-v)\geq0$, and therefore, by \eqref{eq:A-positive-definite-kappa},
\[
\begin{split}
    \left[(\AA+\FF_2)(u)-(\AA+\FF_2)(v)\right](u-v)&=[\AA(u-v)](u-v)+\left[\FF_2(u)-\FF_2(v)\right](u-v) \\&\geq [\AA(u-v)](u-v)\geq \kappa \|u-v\|^2_{H^1(M)},
\end{split}
\]
namely, $\AA + \FF_2$ is a strongly monotone operator, see \cite[p. 475]{Zeidler}.
Additionally, taking $v = 0$, we see that there exists $K>0$ such that for every $u\in H^1(M)$
\begin{equation}\label{eq:coer}
\left[(\AA+\FF_2)(u)\right](u) \geq \kappa \|u\|^2_{H^1(M)} + [\FF_2(0)](u) \ge  \kappa \|u\|^2_{H^1(M)} - K \|u\|_{H^1(M)},
\end{equation}
so
$\AA+\FF_2$ is coercive (see \cite[p. 472]{Zeidler}).
Hence, from Browder's theorem for monotone operators
(see \cite[Theorem 26.A (c)]{Zeidler}), the inverse $(\AA+\FF_2)^{-1} \colon H^{-1}(M)\to H^1(M)$ exists.

To show that $(\AA+\FF_2)^{-1}$ is increasing, take $h_1,h_2 \in H^{-1}(M)$ such that $h_2\geq h_1$
and put $u_i=(\AA+\FF_2)^{-1}(h_i)\in H^1(M)$. Since
\[
\left[\FF_2(u_1)-\FF_2(u_2)\right](u_1-u_2)^+=\int_M 
\left(f_2(\cdot,u_1(\cdot))-f_2(\cdot,u_2(\cdot))\right)(u_1-u_2)\mathbbm{1}_{\{u_1\geq u_2\}} \, \dvg \geq 0,
\]
$\AA(u^+)(u^-)=0$, and from \eqref{eq:A-positive-definite-kappa}, testing the condition $h_2\geq h_1$ with the function $v=(u_1-u_2)^+\geq 0$, we get
\begin{equation}\label{e:A+F_2_coerc}
\begin{split}
0&\geq h_1(v)-h_2(v) =[\AA(u_1-u_2)](u_1-u_2)^+ + \left[\FF_2(u_1)-\FF_2(u_2)\right](u_1-u_2)^+\\
&\geq [\AA(u_1-u_2)^+](u_1-u_2)^+ - [\AA(u_1-u_2)^-](u_1-u_2)^+ =[\AA(u_1-u_2)^+](u_1-u_2)^+\\
&\geq \kappa\|(u_1-u_2)^+\|^2_{H^1(M)}.
\end{split}
\end{equation}
So $(u_1-u_2)^+=0$, i.e., $u_2\geq u_1$, and the proof is completed.
\end{proof}

We reformulate \eqref{eq:functional_equation} as a fixed point problem
\begin{align*}
u = \Lambda(u) := [(\AA + \FF_2)^{-1} \circ \FF_1](u).
\end{align*}
To get the existence of a fixed point of $\Lambda$, in the next Lemma, we will use a variant of the fixed point theorem from \cite{Carl2002}.

\begin{Lem}\label{lem:existence}
There exists a fixed point $u_0 \in H^1(M)$ of $\Lambda$, namely, $\Lambda(u_0)=u_0$. In particular, $u_0$ is a solution to \eqref{eq:functional_equation}.
\end{Lem}
\begin{proof}
Observe that $(H^1(M), \leq)$ is a Banach semilattice (see \cite[Definition 2.1, Lemma 2.1]{Carl2002}) and a  Hilbert space, hence it is reflexive. The operator $\Lambda \colon H^1(M)\to H^1(M)$ is increasing. We will show that there exists a radius $r > 0$ such that
$$\Lambda (H^1(M)) \subset
\left\{ u \in H^1(M) \ : \ \|u\|_{H^1(M)} \leq r \right\}.
$$
For this purpose, take $v \in H^1(M)$ and let $u = \Lambda(v) \in H^1(M)$. Then, from \eqref{eq:coer}, since $B, |C|_g, |D|_g  \in L^{\infty}(M)$,
\begin{align*}
   \kappa \|u\|_{H^1(M)}^2 - K \|u\|_{H^1(M)} &\leq  [(\AA+\FF_2)(u)](u) =  [\FF_1(v)](u) \\
    &= \int_M  f_1(\cdot, v(\cdot)) u \, \dvg  \leq \int_M  |f_1(\cdot, \Theta)| |u| \, \dvg \\
    &\leq \|f_1(\cdot,\Theta)\|_{L^\infty(M)} \|u\|_{L^1(M)} \leq K \|u\|_{H^1(M)}.
\end{align*}
Therefore $\|u\|_{H^1(M)} \leq r$ for some $r > 0$. Then, by \cite[Corollary 2.2]{Carl2002}, there exists $u_0 \in H^1(M)$ such that $\Lambda(u_0)=u_0$.
\end{proof}

\begin{Lem}\label{lem:bound}
Every solution $u \in H^1(M)$ of \eqref{eq:functional_equation} satisfies $\vartheta \leq u \leq \Theta$. In particular, $u_0 \in L^\infty (M)$ and is a weak solution to \eqref{eq:1}.
\end{Lem}

\begin{proof}
Let $u \in H^1(M)$ be any solution of \eqref{eq:functional_equation}. Then
\begin{align*}
[(\AA+\FF_2)(u)] (u-\Theta)^+  = [\FF_1(u)]  (u-\Theta)^+ .
\end{align*}
Since $\Theta$ is a weak supersolution of \eqref{eq:1}, and so also a weak supersolution of \eqref{eq:truncated}, we have
$$0\geq-[(\AA+\FF_2)(\Theta)](u-\Theta)^+ +[\FF_1(\Theta)] (u-\Theta)^+ .$$
Hence, using the fact that $\FF_1(u)=\FF_1(\Theta)$ when $u\geq \Theta$ and arguing as in \eqref{e:A+F_2_coerc}, we obtain
\begin{align*}
    0&=[(\AA+\FF_2)(u)] (u-\Theta)^+  - [\FF_1(u)]  (u-\Theta)^+ \\
    &\geq [(\AA+\FF_2)(u)-(\AA+\FF_2)(\Theta)]  (u-\Theta)^+  - [\FF_1(u)-\FF_1(\Theta)]  (u-\Theta)^+\\
    &= [(\AA+\FF_2)(u)-(\AA+\FF_2)(\Theta)]  (u-\Theta)^+ \geq \kappa \|(u-\Theta)^+ \|^2_{H^1(M)}.
\end{align*}
So we conclude $(u-\Theta)^+=0$, i.e., $u\leq \Theta$. 

The proof of the condition $u\geq\vartheta$ is analogous.
\end{proof}

\begin{proof}[Proof of Theorem \ref{th:Main}]
The first part of the statement follows directly from Lemmas \ref{lem:existence} and \ref{lem:bound}.

Recall that the solution $u$ satisfies $0 < \vartheta \leq u \leq \Theta$. Observe that $u$ is a weak solution of 
\[
-\Delta_g u=\chi u,
\]
where $\chi := - (h-|C|_g^2) + |D|_g^2 u^{2 \cdot 2^*}  + 2 \langle C,D \rangle_g u^{2^*}  - B u^{2^*-2}  + \frac{A}{u^{2^*+2}}$.
Since $\vartheta \le u \le \Theta$, we have $f := \chi u \in L^\infty(M)$. As $M$ is compact, it suffices to argue in finitely many coordinate charts, where $-\Delta_g u = f$ reads $-\partial_i\big(\sqrt{|g|}\,g^{ij}\partial_j u\big) = \sqrt{|g|}\, f$ with smooth, uniformly elliptic coefficients. By \cite[Theorem 8.8 and Lemma 9.16]{MR1814364}, $u \in W^{2,q}(M)$ for every $q<\infty$, hence $u \in \cC^{1,\beta}(M)$ for every $\beta\in(0,1)$ by the Morrey--Sobolev embedding. Let $\alpha \in (0,1)$ be a H\"older exponent of $A$, $B$, $|C|_g$, $|D|_g$, $h$, $\langle C, D\rangle_g$. Since $u$ takes values in $[\vartheta,\Theta] \subset (0,\infty)$, where $t\mapsto t^s$ is Lipschitz continuous for every $s \in \R$, all the powers of $u$ appearing in $\chi u$ are $\alpha$-H\"older continuous, and so is $f$. The Schauder estimates \cite[Theorem 9.19]{MR1814364} (see also \cite[p. 271]{Struwe}) yield $u \in \cC^{2,\alpha}(M)$.
\end{proof}

\section{Nonexistence of smooth positive solutions}\label{sec:non}

Throughout this section, we assume that $A,B,|C|_g,|D|_g, h$ are continuous functions such that $A>0$, $\langle C,D\rangle_g\geq 0$ and $|D|_g>0$ on $M$. 

As in \cite[Section 2]{Hebey2008}, assume that the problem \eqref{eq:1} has a positive solution $u_0$ of class $\cC^2$. Since the domain $M$ is compact, there exists $p_0\in M$ such that $u_0(p_0)>0$ is the minimum of $u_0$. Then $\Delta_g u_0(p_0)\geq 0$, so we can write
\[
\begin{split}
(h(p_0)-|C(p_0)|_g^2) \upz \geq& |D(p_0)|_g^2 \upz^{2 \cdot 2^* + 1} + 2 \langle C(p_0) , D(p_0) \rangle_g \upz^{2^*+1}\\
&- B(p_0) \upz^{2^*-1} + \frac{A(p_0)}{\upz^{2^*+1}}\\
\geq& |D(p_0)|_g^2 \upz^{2 \cdot 2^* + 1} - B(p_0) \upz^{2^*-1} + \frac{A(p_0)}{\upz^{2^*+1}}.
\end{split}
\]
We divide by $u_0(p_0)>0$ and substitute 
\[
z=\upz^{\frac{2}{N}\twostar}=\upz^{\twostar-2},\quad
\mathfrak{a}=A(p_0),\quad \mathfrak{b}=B(p_0),\quad \mathfrak{c}=|C(p_0)|_g^2,\quad \mathfrak{d}=|D(p_0)|_g^2, \quad \mathfrak{h}=h(p_0).
\]
Note that
\[
\upz^{2 \cdot \twostar}=z^N\quad\text{and}\quad 
\upz^{-(\twostar+2)}=z^{-\frac{N}{2\cdot\twostar}(\twostar+2)}=z^{-(N-1)}.
\]
Therefore, we can rewrite the inequality above as
\begin{equation}\label{existence_inequality}
\mathfrak{h}-\mathfrak{c} \geq \mathfrak{d} z^N -\mathfrak{b} z+\mathfrak{a} z^{-N+1}.
\end{equation}

If we consider conditions on the functions $A,\,B,\,C,\,D,\, h$ such that \eqref{existence_inequality} cannot hold, then we cannot have classical positive solutions of the problem.

One can easily see that if $B(p)\leq 0$ and $h(p)-|C(p)|_g^2 \leq 0$ for all $p\in M$, then the inequality \eqref{existence_inequality} cannot hold. Note that under these conditions, Theorem \ref{th:Main} does not apply. Indeed, if (\ref{A:essinf_of_h}) holds, then
\[
\essinf_{M} h>\left(\|C\|_{L^\infty}+\|D\|_{L^\infty} t^{\twostar} \right)^2 - \essinf_M B \cdot t^{2^*-2} + \frac{\|A\|_{L^\infty}}{t^{\twostar+2}}
\geq \|C\|_{L^\infty}^2,
\]
so $h(p)-|C(p)|^2_g>0$ for all $p\in M$.

Since the inequality \eqref{existence_inequality} involves several parameters, it is unclear what
an {\em optimal} condition on the functions $A,B,|C|_g,|D|_g, h$ ensuring that \eqref{existence_inequality} cannot hold
could be. Thus, in the following, we provide some {\em global} conditions on the functions $A,B,|C|_g,|D|_g, h$ that imply that \eqref{existence_inequality} cannot hold, i.e., ensure that there are no positive $\cC^2$ solutions.

\begin{Prop}
\label{P:nonexistence} 
Let $(M,g)$ be a smooth closed manifold of dimension $N\geq 3$. Let $A,B,|C|_g,|D|_g, h$ be continuous functions on $M$ such that $A>0$, $\langle C,D\rangle_g\geq 0$ and $|D|_g>0$ on $M$. The equation \eqref{eq:1} does not have $\cC^2$ positive solutions if one of the following conditions holds:
\begin{enumerate}[label=(\arabic*),ref=\arabic*]
    \item \label{NE1} $B(p)\leq 0$ for $p \in M$ and
    \[
\max_{p\in M} \frac{\bigl(h(p)-|C(p)|_g^2\bigr) A(p)^\frac{N-2}{2N-1}}{|D(p)|_g^\frac{2(N-2)}{2N-1} \bigl((2N-1) N^\frac{1}{2N-1} A(p)^\frac{2N-2}{2N-1} |D(p)|_g^\frac{2}{2N-1} - (N-1)^\frac{2N}{2N-1} B(p)\bigr)} < \frac{1}{N^\frac{N+1}{2N-1} (N-1)^\frac{N-1}{2N-1}},
    \]
    \item \label{NE2} $h(p) - |C(p)|_g^2 < 0$ for $p \in M$ and 
    \[
    \max_{p\in M \setminus B^{-1}(0)} \frac{\left(h(p)-|C(p)|_g^2\right)
    |D(p)|_g^{\frac{2(N+1)}{2N-1}} A(p)^{\frac{N-2}{2N-1}}}
    {B(p)^2} < -\frac{(N+1)^{\frac{N+1}{2N-1}}(N-2)^{\frac{N-2}{2N-1}}}{4(2N-1)},
\]
    \item \label{NE3} $h(p) - |C(p)|_g^2 < 0$ for $p \in M$ and 
    \[
\max_{p\in M \setminus B^{-1}(0)}  \frac{\left(h(p)-|C(p)|^2_g\right) |D(p)|_g^\frac{2}{N-1}}{|B(p)|^\frac{N}{N-1}}\leq -\frac{N-1}{N^\frac{N}{N-1}},
\]
    \item \label{NE4}$|D(p)|_g^{\frac{4N}{2N-1}} A(p)^{\frac{2N-2}{2N-1}}-B(p)^2>0 $ for $p\in M$ and
    \[
    \max_{p\in M} \frac{\left(h(p)-|C(p)|_g^2\right)|D(p)|_g^{\frac{2(N+1)}{2N-1}}A(p)^{\frac{N-2}{2N-1}}}
    {|D(p)|_g^{\frac{4N}{2N-1}} A(p)^{\frac{2N-2}{2N-1}}-B(p)^2} \leq \frac{2N-1}{2(N-1)^{\frac{N-1}{2N-1}}N^{\frac{N}{2N-1}}},
    \]
    \item \label{NE5} $B(p)>0$, $|D(p)|_g^{\frac{2(N-1)}{2N-1}+\frac{2}{N-1}} A(p)^{\frac{N}{2N-1}}- B(p)^\frac{N}{N-1}>0$ for $p\in M$ and
    \[
    \max_{p\in M} \frac{\left(h(p)-|C(p)|^2_g\right) |D(p)|_g^\frac{2}{N-1}}{|D(p)|_g^{\frac{2(N-1)}{2N-1}+\frac{2}{N-1}} A(p)^{\frac{N}{2N-1}}- B(p)^\frac{N}{N-1}}
    \leq \frac{2N-1}{2(N-1)^{\frac{N-1}{2N-1}}N^{\frac{N}{2N-1}}}.
    \]
\end{enumerate}
\end{Prop}

\begin{Rem}
Observe that, from a physical point of view, $B \le 0$ means $\frac{N-1}{N}\tau^2 \le 4\nu$, i.e., the constant potential $\nu$, which acts as a positive cosmological constant, dominates the mean curvature, while $h - |C|_g^2 < 0$ means $R(g) < \pi^{-2}|\mathrm{div}_g(\cD W)|_g^2$. Thus, as in the CMC case \cite{Choquet-Bruhat_2007,Hebey2008}, considering \eqref{existence_inequality}, nonexistence occurs when the scalar curvature of the seed metric is too small compared with the cosmological constant or with the energy density $A$ of $\sigma+\cD W$ and $\pi$.
\end{Rem}

\begin{proof}
Let $f(z):=\fd z^N -\fb z+\fa z^{-N+1}$ for $z>0$. Since
\[
\lim_{z\to 0^+} f(z) = 
\lim_{z\to +\infty}f(z) = +\infty
\]
and
\[
f''(z)=N(N-1)\fd z^{N-2}+N(N-1)\fa z^{-N-1}>0,
\]
$f$ has a unique minimum point $\bar{z}$, which satisfies
\begin{equation}\label{derf}
N\fd \bar{z}^{2N-1}-\fb\bar{z}^N-(N-1)\fa=0.
\end{equation}
Thus, using \eqref{derf} in the definition of $f$, we get
\begin{equation} \label{f2}
f(\bar{z})
=\frac{(2N-1)\fa-(N-1)\fb\bar{z}}{N\bar{z}^{N-1}}
\end{equation}
 or
\begin{equation} \label{f3}
f(\bar{z})
=\frac{2N-1}{N-1}\fd\bar{z}^N-\frac{N}{N-1}\fb\bar{z}.
\end{equation}
\textbf{Condition (\ref{NE1}).}
Since $\fb\leq 0$, from \eqref{derf} we have
\[
0 
\geq N \fd \bar{z}^{2N-1}-(N-1) \fa,
\]
namely,
\[
\bar{z}^{2N-1}\leq\frac{(N-1)\fa}{N\fd}.
\]
Note that, since $\fb\leq 0$, the right-hand side of \eqref{f2} is decreasing in $\bar{z}$, therefore
\begin{align*}
f(\bar{z})
&\geq
\frac{(2N-1)\fa-(N-1)\fb\left(\frac{(N-1)\fa}{N\fd}\right)^\frac{1}{2N-1}}{N\left(\frac{(N-1)\fa}{N\fd}\right)^\frac{N-1}{2N-1}}
=
\fd^\frac{N-2}{2N-1}\frac{(2N-1)N^\frac{1}{2N-1}\fa^\frac{2N-2}{2N-1}\fd^\frac{1}{2N-1}-(N-1)^\frac{2N}{2N-1}\fb }{N^\frac{N+1}{2N-1}(N-1)^\frac{N-1}{2N-1}\fa^\frac{N-2}{2N-1}}.
\end{align*}
Thus, by \eqref{existence_inequality},
\[
\fh-\fc\geq \fd^\frac{N-2}{2N-1}\frac{(2N-1)N^\frac{1}{2N-1}\fa^\frac{2N-2}{2N-1}\fd^\frac{1}{2N-1}-(N-1)^\frac{2N}{2N-1}\fb }{N^\frac{N+1}{2N-1}(N-1)^\frac{N-1}{2N-1}\fa^\frac{N-2}{2N-1}},
\]
which is in contradiction with the assumption of this case.

\noindent \textbf{Condition (\ref{NE2}).} We apply Young's inequality  with the weight $\gamma=\frac{N+1}{2N-1}$ to get
\begin{equation*}
\fd z^N -\fb z+\fa z^{-(N-1)}\geq \left(\frac{\fd}{\gamma}\right)^{\gamma}
\left(\frac{\fa}{1-\gamma}\right)^{1-\gamma}z^{\gamma N-(N-1)(1-\gamma)}-\fb z
=\left(\frac{\fd}{\gamma}\right)^{\gamma}
\left(\frac{\fa}{1-\gamma}\right)^{1-\gamma}z^{2}-\fb z.
\end{equation*}
The right-hand side is a quadratic polynomial, so we easily find the lower bound
\begin{equation}\label{E:b>0_Young_estimate}
\fd z^N -\fb z+\fa z^{-(N-1)}\geq -\frac{\fb^2}{4\left(\frac{\fd}{\gamma}\right)^{\gamma}
\left(\frac{\fa}{1-\gamma}\right)^{1-\gamma}}
=-\frac{\fb^2 (N+1)^{\frac{N+1}{2N-1}}(N-2)^{\frac{N-2}{2N-1}}}{4(2N-1)\fd^{\frac{N+1}{2N-1}}\fa^{\frac{N-2}{2N-1}}}.
\end{equation}
As a consequence, by \eqref{existence_inequality}, we have the bound
\begin{equation*}
    \fh-\fc\geq -\frac{\fb^2 (N+1)^{\frac{N+1}{2N-1}}(N-2)^{\frac{N-2}{2N-1}}}{4(2N-1)\fd^{\frac{N+1}{2N-1}}\fa^{\frac{N-2}{2N-1}}}
\end{equation*}
obtaining a contradiction.

\noindent \textbf{Condition (\ref{NE3}).} 
Assume $\fb>0$. 

The right-hand side of \eqref{f3} is increasing in $\bar{z}$ when
\[
\bar{z}\geq \left(\frac{\fb}{(2N-1)\fd}\right)^\frac{1}{N-1}.
\]
Observe that, by \eqref{derf} and since $\fa>0$,
\[
0
< N\fd\bar{z}^{2N-1}-\fb\bar{z}^N,
\]
namely,
\[
\bar{z} > \left(\frac{\fb}{N\fd}\right)^\frac{1}{N-1}.
\]
Since
\[
\left(\frac{\fb}{(2N-1)\fd}\right)^\frac{1}{N-1}
<
\left(\frac{\fb}{N\fd}\right)^\frac{1}{N-1},
\]
then
\begin{equation}\label{E:estim_x}
f(\bar{z}) > \frac{2N-1}{N-1}\fd\left(\frac{\fb}{N\fd}\right)^\frac{N}{N-1}-\frac{N}{N-1}\fb\left(\frac{\fb}{N\fd}\right)^\frac{1}{N-1}
=
-\frac{N-1}{N^\frac{N}{N-1}}
\frac{\fb^\frac{N}{N-1}}{\fd^\frac{1}{N-1}}.
\end{equation}
Hence, from \eqref{existence_inequality},
\[
\frac{(\fh-\fc)\fd^\frac{1}{N-1}}{|\fb|^\frac{N}{N-1}}
=\frac{(\fh-\fc)\fd^\frac{1}{N-1}}{\fb^\frac{N}{N-1}}
> -\frac{N-1}{N^\frac{N}{N-1}},
\]
which is impossible under the assumptions of this case.\\
If $\fb\leq 0$, then from \eqref{existence_inequality} we have
\[
\fh-\fc\geq \fd z^N -\fb z+\fa z^{-N+1} \geq \fd z^N +\fa z^{-N+1} \geq 0,
\]
and we get a contradiction.

\noindent \textbf{Condition (\ref{NE4}).} 
Let $\alpha=\frac{N-1}{2N-1}$. Applying Young's inequality, we have
\begin{equation*}
\begin{split}
\fd z^N+\fa z^{-N+1}&=\alpha\left(\frac \fd \alpha z^N\right)+(1-\alpha)\left(\frac{\fa}{1-\alpha}z^{-N+1}\right)\geq \left(\frac \fd\alpha z^N\right)^{\alpha}\cdot \left(\frac{\fa}{1-\alpha}z^{-N+1}\right)^{1-\alpha}\\
&=\frac{\fd^{\alpha}\fa^{1-\alpha}}{\alpha^{\alpha}(1-\alpha)^{1-\alpha}}
z^{N\alpha-(N-1)(1-\alpha)}=\frac{(2N-1)\fd^{\frac{N-1}{2N-1}} \fa^{\frac{N}{2N-1}}}{(N-1)^{\frac{N-1}{2N-1}}N^{\frac{N}{2N-1}}}.
\end{split}
\end{equation*}
Combining this estimation with \eqref{E:b>0_Young_estimate} we get
\begin{equation}\label{E:case3_estimation}
\begin{split}
\fd z^N -\fb z+\fa z^{-N+1}&=\frac 12\left(\fd z^N + \fa z^{-N+1}\right)+\left(\frac{\fd}{2}z^N -\fb z+\frac{\fa}{2} z^{-N+1}\right)\\
&\geq \frac{(2N-1)\fd^{\frac{N-1}{2N-1}} \fa^{\frac{N}{2N-1}}}{2(N-1)^{\frac{N-1}{2N-1}}N^{\frac{N}{2N-1}}}-\frac{\fb^2 (N+1)^{\frac{N+1}{2N-1}}(N-2)^{\frac{N-2}{2N-1}}}{4(2N-1)\left(\frac 12 \fd\right)^{\frac{N+1}{2N-1}}\left(\frac 12 \fa\right)^{\frac{N-2}{2N-1}}}\\
&=\frac{(2N-1)}{2(N-1)^{\frac{N-1}{2N-1}}N^{\frac{N}{2N-1}}} \fd^{\frac{N-1}{2N-1}} \fa^{\frac{N}{2N-1}}\\
&\qquad
-\frac{(N+1)^{\frac{N+1}{2N-1}}(N-2)^{\frac{N-2}{2N-1}}}{2(2N-1)}\cdot \frac{\fb^2}{\fd^{\frac{N+1}{2N-1}}\fa^{\frac{N-2}{2N-1}}}.
\end{split}
\end{equation}
By Young's inequality, for positive numbers $x\neq y$ we have 

\[
x^{\frac{x}{x+y}}y^{\frac{y}{x+y}}
< \frac{x^2}{x+y}+\frac{y^2}{x+y}
< x+y,
\]
therefore
\[
\frac{2N-1}{(N-1)^{\frac{N-1}{2N-1}}N^{\frac{N}{2N-1}}}> \frac{(N+1)^{\frac{N+1}{2N-1}}(N-2)^{\frac{N-2}{2N-1}}}{2N-1}.
\]
We can continue the estimation \eqref{E:case3_estimation} with
\[
\begin{split}
\fd z^N -\fb z+\fa z^{-N+1} &> \frac{(2N-1)}{2(N-1)^{\frac{N-1}{2N-1}}N^{\frac{N}{2N-1}}}
\left(\fd^{\frac{N-1}{2N-1}} \fa^{\frac{N}{2N-1}}-\frac{\fb^2}{\fd^{\frac{N+1}{2N-1}} \fa^{\frac{N-2}{2N-1}}}\right).
\end{split}
\]
Finally, we get the following bound
\begin{equation*}
    \fh-\fc > \frac{(2N-1)}{2(N-1)^{\frac{N-1}{2N-1}}N^{\frac{N}{2N-1}}}\cdot\frac{ \fd^{\frac{2N}{2N-1}}\fa^{\frac{2N-2}{2N-1}}-\fb^2}{\fd^{\frac{N+1}{2N-1}}\fa^{\frac{N-2}{2N-1}}},
\end{equation*}
in contrast with the assumptions.

\noindent \textbf{Condition (\ref{NE5}).}
Arguing as above but using \eqref{E:estim_x} instead of \eqref{E:b>0_Young_estimate} and recalling that $f(\bar{z}) = \min f$, we get
\[
\begin{split}
\fd z^N - \fb z+ \fa z^{-N+1}& \geq\frac 12\left(\fd \bar{z}^N + \fa \bar{z}^{-N+1}\right)+\left(\frac{\fd}{2} \bar{z}^N -\fb \bar{z}+\frac{\fa}{2} \bar{z}^{-N+1}\right)\\
&\geq \frac{(2N-1) \fd^{\frac{N-1}{2N-1}} \fa^{\frac{N}{2N-1}}}{2(N-1)^{\frac{N-1}{2N-1}}N^{\frac{N}{2N-1}}}-\frac{N-1}{N^\frac{N}{N-1}}\cdot\frac{\fb^\frac{N}{N-1}}{\left(\frac{\fd}{2}\right)^\frac{1}{N-1}}.
\end{split}
\]
Since $N\geq 3$, we have
\[
2^{2N-1}(N-1)^{(N-1)^2}
\le (N-1)^{2N-1+(N-1)^2}
=(N-1)^{N^2}
<N^{N^2}.
\]
Then,
\[
2^\frac{1}{N-1}(N-1)^{\frac{N-1}{2N-1}}<N^\frac{N^2}{(N-1)(2N-1)},
\]
or equivalently,
\[
2^\frac{1}{N-1}(N-1)^{\frac{N-1}{2N-1}}N^{\frac{N}{2N-1}}<N^\frac{N}{N-1}.
\]
Hence, since $2N-1>2(N-1)$, we get
\[
\frac{2N-1}{2^\frac{1}{N-1}(N-1)^{\frac{N-1}{2N-1}}N^{\frac{N}{2N-1}}}>\frac{2(N-1)}{N^\frac{N}{N-1}}
\]
from which
\begin{equation*}
\frac{2N-1}{2(N-1)^{\frac{N-1}{2N-1}}N^{\frac{N}{2N-1}}}
>\frac{2^\frac{1}{N-1}(N-1)}{N^\frac{N}{N-1}}.
\end{equation*}
Then, continuing the previous estimation, we obtain
\[
\mathfrak{h}-\mathfrak{c}\geq \mathfrak{d}z^N - \mathfrak{b}z + \mathfrak{a}z^{-N+1}> 
\frac{2N-1}{2(N-1)^{\frac{N-1}{2N-1}}N^{\frac{N}{2N-1}}}
\left(\mathfrak{d}^{\frac{N-1}{2N-1}} \mathfrak{a}^{\frac{N}{2N-1}}-\frac{\mathfrak{b}^\frac{N}{N-1}}{\mathfrak{d}^\frac{1}{N-1}}\right),
\]
which contradicts the assumptions.
\end{proof}

\section*{Acknowledgements}
Bartosz Bieganowski and Daniel Strzelecki are partially supported by the National Science Centre, Poland (Grant no. 2022/47/D/ST1/00487). Pietro d'Avenia and Jacopo Schino are members of GNAMPA (INdAM) and are partially supported by the GNAMPA project {\em Problemi di ottimizzazione in PDEs da mo\-del\-li biologici} (CUP E5324001950001).
Pietro d'Avenia is financed by European Union - Next Generation EU - PRIN 2022 PNRR P2022YFAJH {\em Linear and Nonlinear PDE's: New directions and Applications} and is also supported by the Italian Ministry of University and Research under the Program Department of Excellence L. 232/2016 (CUP D93C23000100001).
This paper has been completed during visits to the Politecnico di Bari and the University of Warsaw. The authors thank both institutions for the warm hospitality.

\bibliography{references}

@article{Carl2002,
title = {Elliptic problems with lack of compactness via a new fixed point theorem},
journal = {Journal of Differential Equations},
volume = {186},
number = {1},
pages = {122-140},
year = {2002},
issn = {0022-0396},
doi = {https://doi.org/10.1016/S0022-0396(02)00030-X},
url = {https://www.sciencedirect.com/science/article/pii/S002203960200030X},
author = {S. Carl and S. Heikkilä}
}

@misc{BK,
      title={Note on elliptic equations on closed manifolds with singular nonlinearities}, 
      author={Bartosz Bieganowski and Adam Konysz},
      year={2025},
      eprint={2504.21339},
      archivePrefix={arXiv},
      primaryClass={math.AP},
      url={https://arxiv.org/abs/2504.21339}, 
      note={arXiv:2504.21339, to appear in \textit{Rivista di Matematica della Universit\`a di Parma}.}
}

@article{Choquet-Bruhat_2007,
doi = {10.1088/0264-9381/24/4/004},
url = {https://dx.doi.org/10.1088/0264-9381/24/4/004},
year = {2007},
publisher = {},
volume = {24},
number = {4},
pages = {809},
author = {Choquet-Bruhat, Yvonne and Isenberg, James and Pollack, Daniel},
title = {The constraint equations for the {E}instein-scalar field system on compact manifolds},
journal = {Classical and Quantum Gravity}
}

@article{Hebey2008,
  author    = {Emmanuel Hebey and Frank Pacard and Daniel Pollack},
  title     = {A Variational Analysis of {E}instein–Scalar Field {L}ichnerowicz Equations on Compact {R}iemannian Manifolds},
  journal   = {Communications in Mathematical Physics},
  year      = {2008},
  volume    = {278},
  number    = {1},
  pages     = {117--132},
  doi       = {10.1007/s00220-007-0377-1},
  url       = {https://doi.org/10.1007/s00220-007-0377-1},
  issn      = {1432-0916}
}

@book {Zeidler,
    AUTHOR = {Zeidler, Eberhard},
     TITLE = {Nonlinear functional analysis and its applications. {II}/{B}},
      NOTE = {},
 PUBLISHER = {Springer-Verlag, New York},
      YEAR = {1990},
     PAGES = {i--xvi and 469--1202},
      ISBN = {0-387-97167-X},
   MRCLASS = {47-02 (35-01 35J60 47Hxx 58-01 65Jxx)},
  MRNUMBER = {1033498},
MRREVIEWER = {Jean\ Mawhin},
       DOI = {10.1007/978-1-4612-0985-0},
       URL = {https://doi.org/10.1007/978-1-4612-0985-0},
}

@article{Lichnerowicz,
    AUTHOR = {Lichnerowicz, Andr\'e},
     TITLE = {L'int\'egration des \'equations de la gravitation relativiste
              et le probl\`eme des {$n$} corps},
   JOURNAL = {J. Math. Pures Appl. (9)},
  FJOURNAL = {Journal de Math\'ematiques Pures et Appliqu\'ees. Neuvi\`eme
              S\'erie},
    VOLUME = {23},
      YEAR = {1944},
     PAGES = {37--63},
      ISSN = {0021-7824,1776-3371},
   MRCLASS = {83.0X},
  MRNUMBER = {14298},
MRREVIEWER = {M.\ Wyman},
}

@article {Isenberg95,
    AUTHOR = {Isenberg, James},
     TITLE = {Constant mean curvature solutions of the {E}instein constraint
              equations on closed manifolds},
   JOURNAL = {Classical Quantum Gravity},
  FJOURNAL = {Classical and Quantum Gravity},
    VOLUME = {12},
      YEAR = {1995},
    NUMBER = {9},
     PAGES = {2249--2274},
      ISSN = {0264-9381,1361-6382},
   MRCLASS = {83C05 (58G16)},
  MRNUMBER = {1353772},
MRREVIEWER = {Gilbert\ Weinstein},
       URL = {http://stacks.iop.org/0264-9381/12/2249},
}

@article {Maxwell-2009,
    AUTHOR = {Maxwell, David},
     TITLE = {A class of solutions of the vacuum {E}instein constraint
              equations with freely specified mean curvature},
   JOURNAL = {Math. Res. Lett.},
  FJOURNAL = {Mathematical Research Letters},
    VOLUME = {16},
      YEAR = {2009},
    NUMBER = {4},
     PAGES = {627--645},
      ISSN = {1073-2780},
   MRCLASS = {53C21 (35Q76 53C80 83C05)},
  MRNUMBER = {2525029},
MRREVIEWER = {Fr\'ed\'eric\ Robert},
       DOI = {10.4310/MRL.2009.v16.n4.a6},
       URL = {https://doi.org/10.4310/MRL.2009.v16.n4.a6},
}

@article {Bartnik88,
    AUTHOR = {Bartnik, Robert},
     TITLE = {Remarks on cosmological spacetimes and constant mean curvature
              surfaces},
   JOURNAL = {Comm. Math. Phys.},
  FJOURNAL = {Communications in Mathematical Physics},
    VOLUME = {117},
      YEAR = {1988},
    NUMBER = {4},
     PAGES = {615--624},
      ISSN = {0010-3616,1432-0916},
   MRCLASS = {53C50 (53C42 83C99)},
  MRNUMBER = {953823},
MRREVIEWER = {Judith\ M.\ Arms},
       URL = {http://projecteuclid.org/euclid.cmp/1104161820},
}

@article {Chrusciel,
    AUTHOR = {Chru\'sciel, Piotr T. and Isenberg, James and Pollack, Daniel},
     TITLE = {Initial data engineering},
   JOURNAL = {Comm. Math. Phys.},
  FJOURNAL = {Communications in Mathematical Physics},
    VOLUME = {257},
      YEAR = {2005},
    NUMBER = {1},
     PAGES = {29--42},
      ISSN = {0010-3616,1432-0916},
   MRCLASS = {83C05 (58J45)},
  MRNUMBER = {2163567},
       DOI = {10.1007/s00220-005-1345-2},
       URL = {https://doi.org/10.1007/s00220-005-1345-2},
}

@Inbook{Rendall,
author="Rendall, Alan D.",
title="Mathematical Properties of Cosmological Models with Accelerated Expansion",
bookTitle="Analytical and Numerical Approaches to Mathematical Relativity",
year="2006",
publisher="Springer Berlin Heidelberg",
address="Berlin, Heidelberg",
pages="141--155",
isbn="978-3-540-33484-2",
doi="10.1007/3-540-33484-X_7",
url="https://doi.org/10.1007/3-540-33484-X_7"
}

@article{Rendall2,
title="Accelerated cosmological expansion due to a scalar field whose potential has a positive lower
bound",
author="Rendall, Alan D.",
journal="Class. Quantum Grav.",
year="2004",
volume="21",
number="9",
pages="2445"
}

@book {MR1814364,
    AUTHOR = {Gilbarg, David and Trudinger, Neil S.},
     TITLE = {Elliptic partial differential equations of second order},
    SERIES = {Classics in Mathematics},
      NOTE = {Reprint of the 1998 edition},
 PUBLISHER = {Springer-Verlag, Berlin},
      YEAR = {2001},
     PAGES = {xiv+517},
      ISBN = {3-540-41160-7},
   MRCLASS = {35-02 (35Jxx)},
  MRNUMBER = {1814364},
}

@article {Tarski,
    AUTHOR = {Tarski, Alfred},
     TITLE = {A lattice-theoretical fixpoint theorem and its applications},
   JOURNAL = {Pacific J. Math.},
  FJOURNAL = {Pacific Journal of Mathematics},
    VOLUME = {5},
      YEAR = {1955},
     PAGES = {285--309},
      ISSN = {0030-8730,1945-5844},
   MRCLASS = {06.0X},
  MRNUMBER = {74376},
MRREVIEWER = {B.\ J\'onsson},
       URL = {http://projecteuclid.org/euclid.pjm/1103044538},
}

@book {Struwe,
    AUTHOR = {Struwe, Michael},
     TITLE = {{Variational methods: Applications to nonlinear partial differential equations and Hamiltonian systems}},
    SERIES = {Ergebnisse der Mathematik und ihrer Grenzgebiete. 3. Folge. A
              Series of Modern Surveys in Mathematics },
    VOLUME = {34},
 PUBLISHER = {Springer-Verlag, Berlin},
      YEAR = {2008},
     PAGES = {xx+302},
      ISBN = {978-3-540-74012-4},
   MRCLASS = {49-02 (34C25 35A15 35F20 37J45 47J30 49J10 58E05)},
  MRNUMBER = {2431434},
}

@InProceedings{Isenberg&Bartnik,
author="Bartnik, Robert
and Isenberg, Jim",
editor="Chru{\'{s}}ciel, Piotr T.
and Friedrich, Helmut",
title="The Constraint Equations",
booktitle="The Einstein Equations and the Large Scale Behavior of Gravitational Fields",
year="2004",
publisher="Birkh{\"a}user Basel",
address="Basel",
pages="1--38",
isbn="978-3-0348-7953-8"
}

@article {PremoselliCVPD,
    AUTHOR = {Premoselli, Bruno},
     TITLE = {Effective multiplicity for the {E}instein-scalar field
              {L}ichnerowicz equation},
   JOURNAL = {Calc. Var. Partial Differential Equations},
  FJOURNAL = {Calculus of Variations and Partial Differential Equations},
    VOLUME = {53},
      YEAR = {2015},
    NUMBER = {1-2},
     PAGES = {29--64},
      ISSN = {0944-2669,1432-0835},
   MRCLASS = {58J05 (83C05)},
  MRNUMBER = {3336312},
MRREVIEWER = {Liang\ Zhao},
       DOI = {10.1007/s00526-014-0740-y},
       URL = {https://doi.org/10.1007/s00526-014-0740-y},
}

@article {PremoselliDCDS,
    AUTHOR = {Premoselli, Bruno},
     TITLE = {Einstein-{L}ichnerowicz type singular perturbations of
              critical nonlinear elliptic equations in dimension 3},
   JOURNAL = {Discrete Contin. Dyn. Syst.},
  FJOURNAL = {Discrete and Continuous Dynamical Systems},
    VOLUME = {41},
      YEAR = {2021},
    NUMBER = {11},
     PAGES = {5087--5103},
      ISSN = {1078-0947,1553-5231},
   MRCLASS = {35J60 (35B44 35Q75)},
  MRNUMBER = {4305578},
       DOI = {10.3934/dcds.2021069},
       URL = {https://doi.org/10.3934/dcds.2021069},
}

@article {Chrusciel_Gicquaud,
    AUTHOR = {Chru\'sciel, Piotr T. and Gicquaud, Romain},
     TITLE = {Bifurcating solutions of the {L}ichnerowicz equation},
   JOURNAL = {Ann. Henri Poincar\'e},
  FJOURNAL = {Annales Henri Poincar\'e. A Journal of Theoretical and
              Mathematical Physics},
    VOLUME = {18},
      YEAR = {2017},
    NUMBER = {2},
     PAGES = {643--679},
      ISSN = {1424-0637,1424-0661},
   MRCLASS = {83C05},
  MRNUMBER = {3596773},
MRREVIEWER = {Chopin\ Soo},
       DOI = {10.1007/s00023-016-0501-x},
       URL = {https://doi.org/10.1007/s00023-016-0501-x},
}

@article {Hebey_Veronelli,
    AUTHOR = {Hebey, Emmanuel and Veronelli, Giona},
     TITLE = {The {L}ichnerowicz equation in the closed case of the
              {E}instein-{M}axwell theory},
   JOURNAL = {Trans. Amer. Math. Soc.},
  FJOURNAL = {Transactions of the American Mathematical Society},
    VOLUME = {366},
      YEAR = {2014},
    NUMBER = {3},
     PAGES = {1179--1193},
      ISSN = {0002-9947,1088-6850},
   MRCLASS = {58J05 (83C05)},
  MRNUMBER = {3145727},
MRREVIEWER = {Willie\ W.\ Wong},
       DOI = {10.1090/S0002-9947-2013-05790-X},
       URL = {https://doi.org/10.1090/S0002-9947-2013-05790-X},
}
\bibliographystyle{abbrv}
\end{document}